\documentclass[12pt]{amsart}
\usepackage{tikz-cd}
\usepackage{enumitem}
\setlist[description]{leftmargin=\parindent,labelindent=\parindent}
\usetikzlibrary{matrix,arrows,decorations.pathmorphing}
\usepackage{amssymb}
\usepackage{amsmath,amscd}
\usepackage{mathrsfs}
\usepackage{url}
\usepackage{cite}
\usepackage{fullpage}
\usepackage{hyperref}
\usepackage{verbatim}
\usepackage{comment}
\usepackage{fancyvrb}
\usepackage{fvextra}
\usepackage{caption}
\usepackage{ytableau}

\newtheorem{thm}{Theorem}[section]
\newtheorem{prop}[thm]{Proposition}
\newtheorem{lem}[thm]{Lemma}
\newtheorem{cor}[thm]{Corollary}

\theoremstyle{definition}

\newtheorem{example}[thm]{Example}
\newtheorem{rem}[thm]{Remark}

\numberwithin{equation}{section}

\newcommand{\Z}{\mathbb{Z}}

\newcommand{\zz}{\mathbb{Z}}
\newcommand{\cc}{\mathbb{C}}
\renewcommand{\ss}{\mathbb{S}}
\newcommand{\qq}{\mathbb{Q}}

\newcommand{\M}{\mathcal{M}}

\newcommand{\E}{\mathcal{E}}

\renewcommand{\tilde}{\widetilde}

\DeclareMathOperator{\Aut}{Aut}

\DeclareMathOperator{\SL}{SL}

\DeclareMathOperator{\Sym}{Sym}

\DeclareMathOperator{\tr}{tr}

\DeclareMathOperator{\sign}{sign}

\usepackage{wrapfig}

\newcommand{\Mb}{\overline{\M}}

\title{The eleventh cohomology group of $\Mb_{g,n}$}
\author{Samir Canning}
\author{Hannah Larson}
\author{Sam Payne}
\thanks{
S.C. was partially supported by NSF RTG grant DMS-1502651.
This research was partially conducted during the period H.L served as a Clay Research Fellow. S.P. was supported in part by NSF grants DMS-2001502 and DMS-2053261}

\begin{document}

\maketitle

\begin{abstract}
We prove that the rational cohomology group $H^{11}(\Mb_{g,n})$ vanishes unless $g = 1$ and $n \geq 11$.  We show furthermore that $H^k(\Mb_{g,n})$ is pure Hodge--Tate for all even $k \leq 12$ and deduce that $\# \Mb_{g,n}(\mathbb{F}_q)$ is surprisingly well approximated by a polynomial in $q$. In addition, we use $H^{11}(\Mb_{1,11})$ and its image under Gysin push-forward for tautological maps to produce many new examples of moduli spaces of stable curves with nonvanishing odd cohomology and non-tautological algebraic cycle classes in Chow cohomology.
\end{abstract}

\section{Introduction}
 The Langlands program makes a number of striking predictions about the Hodge structures and Galois representations that appear in the cohomology of moduli spaces of stable curves; see \cite[Section~1.2]{ChenevierRenard15} and \cite{BergstromFaberPayne, BergstromFaber}. While the conjectured correspondence with algebraic cuspidal automorphic representations of conductor 1 remains out of reach, these representations have been classified up to weight 22 \cite{ChenevierLannes}, and some of the resulting  predictions can now be verified unconditionally.  Bergstr\"om, Faber and the third author recently proved that $H^k(\Mb_{g,n})$ vanishes for all odd $k \leq 9$ and all $g$ and $n$ \cite{BergstromFaberPayne}.
For $k = 11$, the conjectural correspondence predicts that $H^{11}(\Mb_{g,n})$ is isomorphic to a direct sum of copies of $H^{11}(\Mb_{1,11})$ and hence should vanish in all cases where $\Mb_{g,n}$ is unirational.
We confirm this prediction unconditionally and show that $H^{11}(\Mb_{g,n})$ vanishes in an even wider range of cases.

\smallskip

\begin{thm}\label{H11}
The cohomology group $H^{11}(\Mb_{g,n})$ is nonzero if and only if $g = 1$ and $n \geq 11$. 
\end{thm}

\smallskip

For $g = 1$ and $n \geq 11$, $H^{11}(\Mb_{1,n})$ is isomorphic to a direct sum of $\binom{n-1}{10}$ copies of $H^{11}(\Mb_{1,11})$ \cite{Getzler}; in particular, it decomposes as $H^{11,0}(\Mb_{1,n}) \oplus H^{0,11}(\Mb_{1,n})$. We show that  $H^{11,0}(\Mb_{1,n})$ is generated by the pullbacks of the distinguished generator of $H^{11,0}(\Mb_{1,11})$, which corresponds to the weight 12 cusp form $\Delta$ for $\SL_2(\Z)$, under the $\binom{n}{11}$ forgetful maps, and describe the relations among these generators.  In this way, we show that $H^{11,0}(\Mb_{1,n})$ is an irreducible $\ss_{n}$-representation isomorphic to the Specht module $V_{n-10,1^{10}}$. 

Next, we address the Hodge structures and Galois representations that appear in other low degrees.  The Langlands program predicts that the cohomology of $\Mb_{g,n}$ should be pure Hodge--Tate in all even degrees less than or equal to 20. This prediction was previously confirmed only in the cases where these cohomology groups are known to be generated by tautological classes, e.g. for $g \leq 2$ \cite{Keel, Getzler, Petersen}, 
for $k \leq 2$ \cite{ArbarelloCornalba}, and for $k = 4$ and $g \geq 8$ \cite{Polito}.  Our second result extends the confirmation of this prediction to a much wider range of cases. The proof is a double induction on $g$ and $n$. The base cases are given by recent results of the first two authors, who  showed that $H^*(\Mb_{g,n})$ is tautological for $g \geq 3$ and $2g + n \leq 14$ \cite{CL-CKgP}.

\smallskip

\begin{thm} \label{even}
For any even $k \leq 12$, the cohomology group $H^{k}(\Mb_{g,n})$ is pure Hodge--Tate.
\end{thm}

\noindent It remains an open problem whether $H^k(\Mb_{g,n})$ is generated by algebraic cycle classes for even $4 \leq k \leq 12$.

\smallskip

As an application of these two theorems, we show that the point count $\#\Mb_{g,n}(\mathbb{F}_q)$ is surprisingly well-approximated by a polynomial in $q$.

\smallskip

\begin{cor} \label{almost}
Assume $g \geq 2$, and let $d = \dim \Mb_{g,n}$. Then 
\[ \#\Mb_{g,n}(\mathbb{F}_q) = \sum_{i=0}^6 h^{2i}(\Mb_{g,n}) q^{d-i} + O(q^{d-13/2}).\]
\end{cor}

\begin{rem}
The point count $\#\Mb_{0,n}(\mathbb{F}_q)$ is a polynomial in $q$, for all $n$, as is $\# \Mb_{1,n}(\mathbb{F}_q)$ for $n \leq 10$.  For $n \geq 11$, $\# \Mb_{1,n}(\mathbb{F}_q)$ was determined by Getzler \cite{Getzler}; it has an approximation to order $O(q^{n-13/2})$ by a polynomial in $q$ minus the correction term $\binom{n-1}{10} \tau(q)$, where $\tau(q)$ denotes the coefficient of $q$ in the Fourier expansion of the weight $12$ cusp form $\Delta$ for $\SL_2(\zz)$.
\end{rem} 

Unlike the cohomology groups in smaller odd degrees, $H^{13}(\Mb_{g,n})$ is nonvanishing in a wide range of cases, including for large $g$, as are all higher degree odd cohomology groups.  Indeed, Pikaart showed that $H^{13}(\Mb_{g,n})$ is nonvanishing for $n \geq 10$ and $g$ sufficiently large, as is $H^{33}(\Mb_g)$ \cite{Pikaart}. Similar nonvanishing statements in higher degrees follow immediately, by Hard Lefschetz. The bounds on $g$ that come from Pikaart’s method are large, typically in the thousands. For instance, van Zelm computes that Pikaart’s method yields $H^{33}(\Mb_{g,1}) \neq 0$ for $g \geq 8069$.  The bounds for $H^{13}(\Mb_{g,10})$ and $H^{33}(\Mb_{g})$ are not explicitly stated in the literature, but there is substantial evidence that such bounds should be far from optimal. While Pikaart’s constructions prove the existence of non-tautological algebraic cycle classes on $\Mb_g$ for $g \geq 16192$, van Zelm proved that this holds for $g \geq 12$ \cite{vanZelm}. Also, Bergstr\"om and Faber have recently shown that $H^{13}(\Mb_{2,n}) \neq 0$ for $n \geq 10$. They also prove that the nonvanishing of $H^{13}(\Mb_{3,n})$ for $n \geq 10$ follows from conjectural parts of the Langlands correspondence.  Here, we prove the latter statement unconditionally, extend it to all higher genera, and also improve Pikaart’s bound for the nonvanishing of $H^{33}(\Mb_g)$.

\smallskip

\begin{thm}\label{higher}
Assume $k \leq 11$. Let $g_1,\dots,g_k$ be distinct positive integers, and set $g=1 + g_1+\cdots+g_k$. Then  
$$
H^{11+2k}(\Mb_{g,n})\neq 0 \mbox{ \ for } n \geq 11 - k.
$$ 
In particular, $H^{13}(\Mb_{g,n})\neq 0$ for $g\geq 2$ and $n\geq 10$, and $H^{33}(\Mb_g)\neq 0$ for $g\geq 67$.
\end{thm}

\smallskip

As a further application, we prove the existence of non-tautological classes in the Chow rings $A^*(\Mb_{g,n})$ in a number of cases where this was not previously known.

\smallskip

\begin{cor}\label{nontautological}
Consider $\Mb_{g,n}$ as a stack over $\cc$. For any $(g,n)$ as in Theorem \ref{higher}, the quotient $A^*(\Mb_{g,n})/R^*(\Mb_{g,n})$ is uncountable, as is the subgroup of $A^*(\Mb_{g,n})$ generated by cycles algebraically equivalent to zero.
\end{cor}

\smallskip

\noindent This provides many new examples of $(g,n)$ for which  $A^*(\Mb_{g,n})$ is not tautological. In particular, the existence of non-tautological Chow classes is new for $(4,n)$ with $9\leq n\leq 16$; for $(5,n)$ with $9\leq n\leq 13$; for $(6,n)$ with $9\leq n\leq 11$; and for $(7,8)$ and $(7,9)$. 

Existence results for non-tautological classes come in two flavors. There are cases where one can write down \emph{explicit} examples of non-tautological Chow classes. Graber and Pandharipande gave the first such example when $(g,n)=(2,20)$ \cite{GraberPandharipande}. Van Zelm generalized their example to show the existence of explicit non-tautological Chow classes on $\Mb_{g,n}$ for $2g+n\geq 24$ and $g\geq 2$ \cite{vanZelm}. A nice feature of these examples is that they are non-tautological in both Chow and cohomology. There are also the \emph{inexplicit} non-tautological Chow classes, which arise from the existence of odd cohomology. The first such examples are for $g=1, n\geq 11$, where the existence of a holomorphic 11-form implies that $A_{0}(\Mb_{1,n})$ is infinite dimensional \cite{Roitman}. Bergstr\"om and Faber showed that there is odd cohomology on $\Mb_{2,n}$ for $n\geq 10$
\cite{BergstromFaber}, implying that there are non-tautological Chow classes as well, by results of Kimura and Totaro \cite{Kimura, TotaroCKgP}; see Theorem \ref{cycleclass}. The examples provided by Corollary~\ref{nontautological} are also of this inexplicit form. We do not know whether all of the non-tautological classes in these inexplicit cases are homologically equivalent to zero.

\subsection{Methods}
Arbarello and Cornalba introduced an inductive method for studying cohomology groups of $\Mb_{g,n}$ and applied this to prove the vanishing of $H^k(\Mb_{g,n})$ for $k \in \{1,3,5\}$ \cite{ArbarelloCornalba}. The same method was used to prove vanishing for $k \in \{7,9\}$ after establishing the additional base cases needed to run the induction, via point counting over finite fields \cite{BergstromFaberPayne}.  Our proof of Theorems~\ref{H11} and \ref{even} start from the observation that the same induction can be used to control the Hodge structures and Galois representations that appear in $H^k(\Mb_{g,n})$ even when these groups do not vanish.  The first two authors recently established the base cases needed for $k \leq 12$ \cite{CL-CKgP}. Running this induction when $k$ is even leads directly to Theorem~\ref{even}.Doing so for $k = 11$ shows that $H^{11}(\Mb_{g,n})$ injects into a direct sum of copies of $H^{11}(\Mb_{1,11})$ (Proposition~\ref{s12}).  This is enough to confirm the prediction that $H^{11}(\Mb_{g,n})$ vanishes whenever $\Mb_{g,n}$ is unirational, but a different argument is needed to prove that it vanishes whenever $g \geq 2$.

The Arbarello--Cornalba induction uses the excision sequence for the pair of $\Mb_{g,n}$ with its boundary $\partial \M_{g,n}:=\Mb_{g,n}\smallsetminus \M_{g,n}$, along with the map
\begin{equation}\label{map1}
H^{k}(\Mb_{g,n})\rightarrow H^{k}(\tilde{\partial \M}_{g,n}).
\end{equation}
given by pullback to the normalization of the boundary.  Our proof of Theorem~\ref{H11} uses the observation that \eqref{map1} is the first arrow in a natural chain complex whose $j$th term may be identified with the cohomology of the normalization of the closure of the codimension $j$ boundary strata with coefficients in a natural local system, the determinant of the permutation representation on the branches of the boundary divisor.  This complex has several natural interpretations: it is the $k$th weight-graded piece of the Feynman transform of the modular operad that takes the value $\qq$ for every $(g,n)$ \cite{GetzlerKapranov}. It is also the weight $k$ row in the $E_1$-page of a natural spectral sequence obtained via Poincar\'e duality from Deligne's weight spectral sequence for the pair $(\Mb_{g,n}, \partial \M_{g,n})$ and has a natural interpretation as a decorated graph complex \cite[Section~2.3]{PayneWillwacher21}.

For $k = 11$, we examine the first two maps in this complex.  Assuming that  $H^{11}(\Mb_{2,n})$ vanishes for all $n$, a double induction on $g$ and $n$ shows that $H^{11}(\Mb_{g,n})$ vanishes whenever $g > 2$. In Section~\ref{sec:g=2}, we prove the needed base cases, i.e. the vanishing of $H^{11}(\Mb_{2,n})$ for all $n$, by explicit calculations using the generators and relations for $H^{11}(\Mb_{1,n})$. It should also be possible to deduce these base cases from results of Petersen \cite{Petersen,Petersenlocalsystems}; see Remark~\ref{alternate2}.

\subsection{Structure of the paper}
In Section \ref{genus1}, we recall how to extract $H^{11}(\Mb_{1,n})$ with its Hodge structure or Galois representation from the work of Getzler \cite{Getzler}. We then describe generators and relations for this group and describe the $\ss_n$-action and the pullback under tautological morphisms in terms of these generators. In Section \ref{original}, we recall the inductive method of Arbarello and Cornalba and use it to prove Theorem \ref{even}. In Section \ref{improved}, we present the inductive argument for vanishing of $H^{11}(\Mb_{g,n})$ for $g > 2$, using the weight spectral sequence, assuming the vanishing for $g = 2$. We then prove the vanishing in the necessary base cases, for $g = 2$, using the explicit generators and relations for $g = 1$. In Sections \ref{pointcounts}, \ref{higherodd}, and \ref{Chow}, we prove Corollary \ref{almost}, Theorem \ref{higher}, and Corollary \ref{nontautological}, respectively.

\subsection{Notations and conventions}
We denote by $H^{k}(\Mb_{g,n})$ either the $\qq$-Hodge structure $H^{k}((\Mb_{g,n})_{\cc},\qq)$ or the absolute Galois representation $H^{k}_{\text{\'et}}((\Mb_{g,n})_{\overline{\qq}},\qq_{\ell})$. We write $\mathsf{L}$ for the Tate motive.
We say that $H^k(\Mb_{g,n})$ is \emph{pure Hodge--Tate} if it is isomorphic to a direct sum of the Betti or $\ell$-adic realizations of powers of $\mathsf{L}$. We write $\mathsf{S}_{12}$ for the motive associated to the weight $12$ modular form $\Delta$, whose Betti and $\ell$-adic realizations are  $H^{11}(\Mb_{1,11})$. We denote by $A^*(X)$ the Chow ring with rational coefficients of a variety or Deligne--Mumford stack $X$.

\subsection*{Acknowledgments}
We are grateful to Jonas Bergstr\"om, Carel Faber, Dan Petersen, and Burt Totaro for helpful conversations. We especially thank Burt Totaro for suggesting an improvement of an earlier verison of Theorem \ref{cycleclass}, and thus Corollary \ref{nontautological}, and Dan Petersen for suggesting the proof sketch in Remark \ref{alternate1}. We thank the referee for helpful comments on an earlier version of this article.

\section{Genus $1$}\label{genus1}
In this section, we give explicit generators and relations for $H^{11}(\Mb_{1,n})$, and describe the $\ss_n$-action and the pullback to boundary divisors in terms of these generators.  These formulas will be used in  Section \ref{in2}, in our proof that $H^{11}(\Mb_{2,n}) = 0$.

\subsection{Dimension and Hodge structure} 
We start by explaining how to extract $H^{11}(\Mb_{1,n})$ with its Hodge structure or Galois representation from \cite{Getzler}.
Getzler gives generating functions
that determine
the cohomology groups $H^{*}(\Mb_{1,n})$ with their $\ss_n$-actions. 
These formulas simplify substantially when forgetting the $\ss_n$-action, so we begin by using Getzler's formula to extract $H^{11}(\Mb_{1,n})$ non-equivariantly.   
Below, as in \cite{Getzler}, we write $\mathsf{L}$ for the Tate motive, and $\mathsf{S}_{2k+2}$ for the Hodge structure associated to the space of cusp forms of weight $2k+2$ (see \cite[p. 489]{Getzler} for definition). We note that $\mathsf{S}_{2k+2} = 0$ for $k \leq 4$.

\begin{lem} \label{firstid}
The cohomology group  $H^{11}(\Mb_{1,n})$ is a direct sum of ${n-1 \choose 10}$ copies of $\mathsf{S}_{12}$.
\end{lem}
\begin{proof}
Let $\mathsf{e}^{\ss_n}(\M)$ denote the $\ss_n$-equivariant Euler characteristic of $\M$ in the Grothendieck ring of equivariant mixed Hodge structures.
Getzler defines two families of generating functions $\mathbf{a}_i = \sum \mathsf{e}^{\ss_n}(\M_{i,n})$ and $\mathbf{b}_i = \sum \mathsf{e}^{\ss_n}(\Mb_{i,n})$. For $i = 0$ or $1$ these generating functions are power series in the Hodge structures
$\mathsf{L}$ and $\mathsf{S}_{2k+2}$ whose coefficients are symmetric functions.

To get to the ordinary Euler characteristic generating function  $\mathsf{e}(\M)$ from $\mathsf{e}^{\ss_n}(\M)$, we apply
Getzler's $\mathrm{rk}$ functor, which is defined by setting the power sums to $p_1 = x$ and $p_n = 0$ for $n > 0$.  It sends $\mathsf{e}^{\ss_n}(\Mb_{i,n})$ to $\mathsf{e}(\Mb_{i,n}) \frac{x^n}{n!}$ \cite[p. 484]{Getzler}. We shall write $a_i = \mathrm{rk}(\mathbf{a}_i)$ and $b_i = \mathrm{rk}(\mathbf{b}_i)$, which are power series in $\mathsf{L}$ and $\mathsf{S}_{2k+2}$ when $i = 0$ or $1$.
For example,
\[b_0 = \mathrm{rk}(\mathbf{b}_0) = \sum_{n \geq 3} \mathsf{e}(\Mb_{0,n}) \frac{x^n}{n!}. \]
Important for us is that
\[1 + x + b_0' = e^x + \text{terms divisible by $\mathsf{L}$}.\]
For $b_1$, we are interested in the coefficient of $\frac{x^n}{n!}\mathsf{S}_{12}$, which is equal to the negative of the multiplicity of $\mathsf{S}_{12}$ in $H^{11}(\Mb_{1,n})$ by construction.
Applying $\mathrm{rk}$ to \cite[Theorem 2.5]{Getzler} relates $b_1$ to $a_1$. Note that the symbol $\circ$ in \cite[Theorem 2.5]{Getzler} denotes plethysm of symmetric functions; applying $\mathrm{rk}$ turns this plethysm into composition of functions, as can be seen from the properties characterizing plethysm in \cite[Section 5.2]{Getzler2}. In particular, we find
\begin{equation} \label{Gthm} b_1 = a_1(x + b_0') + \text{terms built from $a_0$}.
\end{equation}
The terms built from $a_0$ cannot contribute to the coefficient of $\mathsf{S}_{12}$. (We note that there is a small error in \cite[Theorem 2.5]{Getzler}, which is corrected in \cite[p. 306]{ConsaniFaber}, but it occurs in these terms built from $a_0$, and thus will not affect the outcome of our calculation.)

To expand the right-hand side of \eqref{Gthm}, apply $\mathrm{rk}$ to the equation for $\mathbf{a}_1$ in \cite[p. 489]{Getzler}, and then plug in $x + b_0'$ for $x$ (so substitute $p_1 = x+b_0'$ and $p_n = 0$ for $n > 1$):
\[a_1 (x + b_0') = \mathrm{res}_0\left[
\frac{(1 + x + b_0')^{1 - \omega -\mathsf{L}/\omega + \mathsf{L}}-1}{1 - \omega -\mathsf{L}/\omega + \mathsf{L}} \times \left(\sum_{k=1}^\infty \left(\frac{\mathsf{S}_{2k+2} + 1}{\mathsf{L}^{2k+1}}\right) \omega^{2k} - 1\right)(\omega - \mathsf{L}/\omega )d\omega \right]. \]
In the middle parenthesized term, $\mathsf{S}_{12}$ is multiplied by $\omega^{10}/\mathsf{L}^{11}$. Since we need to take the residue at $0$ with respect to $\omega$, the coefficient of $\mathsf{S}_{12}$
appears when the other terms combine to give $\mathsf{L}^{11}/\omega^{11}$.
To get $\mathsf{L}^{11}/\omega^{11}$ we must use the $-\mathsf{L}/\omega$ piece of the $(\omega - \mathsf{L}/\omega)$ term. Similarly, when we expand the first term, only the powers of $\mathsf{L}/\omega$ are relevant. From this, we see that the coefficient of $\mathsf{S}_{12}$ in the above display is the negative of the
coefficient of $y^{10}$ in
\[\frac{(1 + x + b_0')^{1 + y}}{1+y} = \frac{e^{x(1+y)}}{1 + y} + \langle \mathsf{L} \rangle = \sum_{n = 1}^\infty \frac{x^n}{n!}(y+1)^{n-1} + \langle \mathsf{L} \rangle. \]
In conclusion, $H^{11}(\Mb_{1,n})$ consists of ${n-1 \choose 10}$ copies of $\mathsf{S}_{12}$.
\end{proof}

Getzler's formulas also encode the $\ss_n$-action on $H^{11}(\Mb_{1,n})$. We recover this information in a different way, by describing generators on which the $\ss_n$-action is evident, as follows.

\subsection{Generators and their pullbacks} \label{ga}
To begin, in the case $n = 11$, Lemma \ref{firstid} tells us
\[H^{11}(\Mb_{1,11})\otimes \cc = H^{11,0}(\Mb_{1,11}) \oplus H^{0, 11}(\Mb_{1,11}).\] 
The weight $12$ cusp form of $\SL_2(\zz)$ gives rise to a distinguished generator $\omega \in H^{11,0}(\Mb_{1,11})$; see \cite[p. 14]{FPhandbook} for an explicit geometric construction.
It is evident from this construction (or from \cite{Getzler}) that $\ss_{11}$ acts by the sign representation.

We now describe a natural collection of forms in $H^{11,0}(\Mb_{1,n})$, which we will soon see are generators. These forms come from pulling back the distinguished generator of $H^{11,0}(\Mb_{1,11})$ under the various forgetful maps $\Mb_{1,n} \to \Mb_{1,11}$.
Precisely, given an ordered subset $A \subset \{1, \ldots, n\}$ with $|A| = 11$, write $f_A\colon \Mb_{1,n} \to \Mb_{1,A} \cong \Mb_{1,11}$ for the projection map and define $\omega_A := f_A^*\omega \in H^{11,0}(\Mb_{1,n})$.

The pullbacks of these forms to boundary divisors follow a simple rule. By the K\"unneth formula, the only boundary divisors of $\Mb_{1,n}$ with non-zero $H^{11}$ are those of the form 
\begin{equation} \label{defdb}
D_B = \Mb_{1,B \cup p} \times \Mb_{0, B^c \cup q},
\end{equation}
where $|B| \geq 10$ and $\iota_B\colon D_B \to \Mb_{1,n}$ is the map that glues $p$ to $q$. 
In this case, projection onto the first factor $\mathrm{pr}_1\colon D_B \to \Mb_{1,B \cup p}$ induces an isomorphism 
\begin{equation}\label{firstfactor}
\mathrm{pr}_1^*\colon H^{11}(\Mb_{1,B \cup p}) \xrightarrow{\sim} H^{11}(D_B).
\end{equation}
Given ordered subsets $A$ and $B$ of $\{1, \ldots, n\}$ with $|A| = 11$ and $|A \cap B| = 10$, there is a unique element $i \in A$ such that $i \notin B$. Let $\epsilon(A)$ denote the ordered set obtained from $A$ by replacing $i$ with $p$, so $\epsilon(A)$ is a subset of $B$.
If $|A \cap B| = 11$ so that $A$ is already contained in $B$, then we set $\epsilon(A) = A$.
\begin{lem} \label{pullbacks}
Given a boundary divisor $\iota_B\colon D_B \to \Mb_{1,n}$, and $A \subset \{1, \ldots,n\}$ with $|A| = 11$,
we have
\[\iota_B^*\omega_A = \begin{cases} \mathrm{pr}_1^*\omega_{\epsilon(A)} & \text{if $|A \cap B| \geq 10$} \\ 0 & \text{otherwise.} \end{cases} \]
\end{lem}
\begin{proof}
First suppose $|A \cap B| \leq 9$, so $|B^c \cap A| \geq 2$. Then there is a commutative diagram
\begin{equation} \label{d1}
\begin{tikzcd}
D_B \arrow{r}{\iota_B} \arrow{d} & \Mb_{1,n} \arrow{d}{f_A} \\
\Mb_{1,(A \cap B) \cup p} \times \Mb_{0,(A \cap B^c) \cup q} \arrow{r} & \Mb_{1,A},
\end{tikzcd}
\end{equation}
where the horizontal maps glue $p$ to $q$ and the
vertical maps forget markings not in $A$.
In this case, the image of $f_A\circ \iota_B$ is a proper boundary divisor in $\Mb_{1,A}$, which has no holomorphic $11$-forms. Hence, the pullback of the generator of $H^{11,0}(\Mb_{1,A})$ to $D_B$ vanishes.

Now suppose $|A \cap B| \geq 10$, so $|B^c \cap A| \leq 1$. Then, the lower left-hand side of \eqref{d1} must be replaced by $\Mb_{1, p \cup (B \cap A)}$. Thus, there is another commutative diagram
\begin{equation}
\begin{tikzcd}
D_B \arrow{r}{\iota_B} \arrow{d}[swap]{\mathrm{pr}_1} & \Mb_{1,n} \arrow{d}{f_A} \\
\Mb_{1,B \cup p} \arrow{r}[swap]{f_A} & \Mb_{1,A},
\end{tikzcd}
\end{equation}
where if $|A \cap B| = 10$, 
we identify $p$ with the unique symbol of $A$ not contained in $B$.
\end{proof}

\subsection{Relations and the $\ss_n$-action} \label{sn}
The group $\ss_n$ acts on the subsets $A \subset \{1, \ldots, n\}$ and correspondingly on the subspace of $H^{11,0}(\Mb_{1,n})$ generated by the $\omega_A$.
Note that, for any permutation $\sigma$ in the subgroup of $\ss_n$ fixing $A$, we have $\omega_{\sigma(A)} = \sign(\sigma) \omega_{A}$.

To identify our representation, we briefly recall some of the combinatorial objects that arise in the representation theory of $\ss_n$.

\medskip

A \emph{tabloid} is an equivalence class of tableaux, which identifies tableaux up to reordering rows. Given a tableau $T$, we write $\{T\}$ for the corresponding tabloid. Given a partition $\lambda$ of $n$, we denote by $M_{\lambda}$ the vector space with basis given by tabloids of shape $\lambda$.
    The \emph{Specht module generator} associated to a tableau $T$ is the vector
\begin{equation} \label{spmg} \sum_{\sigma \in \mathrm{C}_T}  \sign(\sigma) \{\sigma(T)\} \in M_{\lambda},
\end{equation}
where $\mathrm{C}_T \subset \ss_n$ is the subgroup that preserves the columns of $T$ setwise.
 The subspace of $M_{\lambda}$ generated by the vectors \eqref{spmg} as $T$ runs over all tableaux is an irreducible representation called the \emph{Specht module} $V_\lambda$. The Specht module generators associated to the standard tableaux on $\lambda$ form a basis for $V_{\lambda}$.

\medskip

To each ordered subset $A \subset \{1, \ldots, n\}$ of size $11$, we associate a tableau $T_A$ of shape $(n-10,1^{10})$ which has the symbols of $A$ in order down the first column and the rest of the first row filled in increasing order.

\begin{prop} \label{therep}
There is an $\ss_n$-equivariant isomorphism $H^{11,0}(\Mb_{1,n}) \xrightarrow{\sim} V_{n-10,1^{10}}$ taking $\omega_A$ to the Specht module generator associated to $T_A$.
\end{prop}

\begin{proof}
The dimension of the Specht module $V_{n-10,1^{10}}$ is the number of standard tableaux of shape $(n-10,1^{10})$, which is ${n-1 \choose 10}$.
Therefore, by Lemma \ref{firstid}, it will suffice to give a $\ss_n$-equivariant map from the subspace of $H^{11,0}(\Mb_{1,n})$ generated by $\{\omega_A\}$  to $V_{n-10,1^{10}}$ that takes $\omega_A$ to the Specht module generator associated to $A$.

We are going to study the image of subspace generated by $\{\omega_A\}$ under the pullback map
\begin{equation} \label{quo} H^{11,0}(\Mb_{1,n}) \to \bigoplus_{\substack{|B| = 10}} H^{11,0}(D_{B}),
\end{equation}
where $B$ is an ordered subset of $\{1, \ldots, n\}$ and $D_B$ is as in \eqref{defdb}.
Let $W_B$ be the element of the right hand side of \eqref{quo} which has component $\mathrm{pr}_1^*\omega \in H^{11,0}(D_B)$ and $0$ in all other components.
The collection of ordered subsets $B$ of size $10$ is in bijection with tabloids on $(n-10,1^{10})$ where $B^c$ fills the row and $B$ (in order) fills the column. Thus, the right-hand side is identified with the vector space of tabloids $M_{n-10,1^{10}}$;
given a tabloid $\{T\}$ corresponding to $B$ we write $W_{\{T\}} = W_B$.

Fix $A \subset \{1, \ldots, n\}$ with $|A| = 11$.
If $B \subset A$, with $|B| = 10$, then $A$ is a permutation of $i \cup B$ for some $i \notin B$. 
Let
$\sigma_{A \to B}\colon i \cup B \to A$ denote the corresponding permutation.
By Lemma \ref{pullbacks}, the image of $\omega_A$ in the right hand side of \eqref{quo} is
\[\omega_A \mapsto \sum_{B \subset A} \sign(\sigma_{B \to A}) W_B = \sum_{\sigma \in S_A} \sign(\sigma) W_{\{\sigma(T_A)\}},\]
which is the Specht module generator defined in \eqref{spmg}, and the proposition follows.
\end{proof}

\begin{cor}
The forms $\{ \omega_A : 1 \in A, \text{ $A$ increasing} \}$ form a basis for $H^{11}(\Mb_{1,n})$.
\end{cor}

\begin{cor} \label{cor:inj}
The pullback map $H^{11,0}(\Mb_{1,n}) \to \bigoplus_{\substack{|B| = 10}} H^{11,0}(D_{B})$ is injective.
\end{cor}

\begin{rem}\label{alternate1}
Dan Petersen suggested an alternate method to obtain several of the results in this section, which avoids the manipulations with generating functions in Lemma \ref{firstid}. We sketch his argument here. Let $f^{1}:\E\rightarrow \M_{1,1}$ denote the universal elliptic curve and $\E^{n}$ denote the $n$-fold fiber product of $\E$ with itself over $\M_{1,1}$. Note that $\M_{1,n}$ is an open substack of $\E^{n-1}$. By the long exact sequences for the pairs $(\Mb_{1,n},\M_{1,n})$ and $(\E^{n-1},\M_{1,n})$, we see that there are natural isomorphisms
\[
W_{11}H^{11}(\E^{n-1})\cong W_{11}H^{11}(\M_{1,n})\cong H^{11}(\Mb_{1,n}).
\]
One can then study the Leray spectral sequence for the smooth morphism $f^{n-1}:\E^{n-1}\rightarrow \M_{1,1}$. Let $\mathbb{V}$ denote the local system $R^1f^1_*\qq$. Then by the K\"unneth formula,
\[
Rf^{n-1}_*\qq\cong (Rf^{1}_*\qq)^{\otimes n-1}\cong (\qq \oplus \mathbb{V}[-1]\oplus \qq(-1)[-2])^{\otimes n-1}.
\]
The pure cohomology $W_{11}H^{11}(\E^{n-1})$ arises from the $\binom{n-1}{10}$ summands $\mathbb{V}[-1]^{\otimes 10}$, each of which gives a copy of $\mathsf{S}_{12}$. This gives Lemma \ref{firstid}. To identify the $\ss_n$ representation as in Proposition \ref{therep}, one first notes that as an $\ss_{n-1}$ representation, we have
\[
W_{11}H^{11}(\E_{n-1})\cong \mathsf{S}_{12}\otimes \mathrm{Ind}_{\ss_{10}\times \ss_{n-11}}^{\ss_{n-1}}(\mathrm{sgn}\boxtimes \mathrm{triv}). 
\]
By the Pieri formula and the branching rule for the symmetric group, it follows that as an $\ss_n$ reprsentation,
\[
W_{11}H^{11}(\E_{n-1})\cong \mathsf{S}_{12} \otimes V_{n-10,1^{10}}.
\]
\end{rem}

\section{Applying the Arbarello--Cornalba induction}\label{original}

We start by recalling the inductive method of Arbarello and Cornalba \cite{ArbarelloCornalba}, by excision of the boundary and pullback to its normalization. We then apply this method to prove Theorem \ref{even} and a preliminary proposition about the degree $11$ cohomology of $\Mb_{g,n}$.

\subsection{Restricting to boundary divisors}
Consider the excision long exact sequence associated to the boundary $\partial \M_{g,n}=\Mb_{g,n}\smallsetminus \M_{g,n}$: 
\begin{equation} \label{eq:boundaryinj}
\cdots\rightarrow H^k_c(\M_{g,n})\rightarrow H^k(\Mb_{g,n})\rightarrow H^k(\partial \M_{g,n})\rightarrow H^{k+1}_c(\M_{g,n})\rightarrow \cdots.
\end{equation}
Note that this sequence is in fact a long exact sequence of mixed Hodge structures or $\ell$-adic Galois representations. In particular, when $H^k_c(\M_{g,n})=0$, there is an injective morphism 
\[
H^k(\Mb_{g,n})\hookrightarrow H^k(\partial \M_{g,n}).
\]
Let $\tilde{\partial \M}_{g,n}$ denote the normalization of $\partial \M_{g,n}$. Arbarello and Cornalba improve on the injectivity of \eqref{eq:boundaryinj}, as follows.
\begin{lem}[Lemma 2.6 of \cite{ArbarelloCornalba}]\label{injective}
Suppose $H^k_c(\M_{g,n})=0$. Then the pullback
\[
H^k(\Mb_{g,n})\rightarrow H^k(\tilde{\partial \M}_{g,n})
\]
is injective.
\end{lem}
For fixed $k$, the following proposition gives vanishing of compactly supported cohomology in all but an explicit finite collection of cases.
\begin{prop}[Proposition 2.1 of \cite{BergstromFaberPayne}]\label{compactsupport}
Assume $g\geq 1$. 
\begin{equation*}
    H^k_c(\M_{g,n})=0 \text{ for } \begin{cases}
    & k<2g \text{ and } n=0,1 \\
    & k<2g-2+n \text{ and } n\geq 2.
    \end{cases}
\end{equation*}
\end{prop}

\subsection{The case of even degrees $k \leq 12$}
Let $RH^*(\Mb_{g,n})\subset H^*(\Mb_{g,n})$ be the tautological cohomology ring. Tautological classes are algebraic and defined over $\mathbb{Z}$, so if $H^k(\Mb_{g,n}) = RH^k(\Mb_{g,n})$, then $H^k(\Mb_{g,n})$ is pure Hodge--Tate. The next lemma provides the necessary base cases for the inductive argument.
\begin{lem} \label{basecases}
If $2g - 2 + n \leq 12$ and $k \leq 12$ is even, then $H^k(\Mb_{g,n})$ is pure Hodge--Tate.
\end{lem}
\begin{proof}
For $g = 0$ and any $n$, all cohomology is tautological \cite{Keel}, as is all \emph{even} cohomology for $g = 1$ \cite{Petersengenus1}, and for $g = 2$ and $n < 20$ \cite{Petersen}. Finally, for $g \geq 3$, all cohomology is tautological for $2g - 2 +n \leq 12$ \cite[Theorem 1.4]{CL-CKgP}. 
\end{proof}

\begin{proof}[Proof of Theorem \ref{even}]
We induct on $g$ and $n$.
By Lemma \ref{basecases}, we can assume $2g-2+n>k$. 
By Lemma \ref{compactsupport}, we have an injection
\begin{equation} \label{inc}
H^{k}(\Mb_{g,n})\hookrightarrow H^{k}(\tilde{\partial \M}_{g,n}).
\end{equation}
Each component of $\tilde{\partial \M}_{g,n}$ is a quotient by a finite group of
$\Mb_{g-1,n+2}$ or $\Mb_{g_1,n_1+1}\times\Mb_{g_2,n_2+1}$, where $g_1+g_2=g$, $n_1 + n_2 = n$, and $2g_i-2+n_i>0$. By induction on $g$, we know $H^k(\Mb_{g-1,n+2})$ is pure Hodge--Tate.
Meanwhile, note that $H^j(\Mb_{g',n'})=0$ for all $(g',n')$ and odd $j \leq k/2 \leq 6$ by \cite{ArbarelloCornalba}. Hence, the K\"unneth formula shows that
\begin{equation}\label{second2}
H^{k}(\Mb_{g_1,n_1+1}\times\Mb_{g_2,n_2+1})=\bigoplus_{i=0}^{k/2}
H^{2i}(\Mb_{g_1,n_1+1})\otimes H^{k-2i}(\Mb_{g_2,n_2+1}).
\end{equation}
Inductively, we know that the right hand side of \eqref{second2} is pure Hodge--Tate. Thus, \eqref{inc} gives an injection of $H^k(\Mb_{g,n})$ into a Hodge structure or Galois representation that is pure Hodge--Tate, and it follows that  $H^k(\Mb_{g,n})$ is pure Hodge--Tate as well.
\end{proof}

\subsection{The case of degree $k=11$}
The base cases required to run an analogous induction for $k = 11$ are those where $H^{11}(\Mb_{g,n})$ does not inject into $H^{11}(\tilde{\partial \M}_{1,11})$. Recent results of the first two authors rule out any such bases cases with $g \geq 2$.

\begin{lem} \label{0to1}
If $g \geq 2$, then $H^{11}(\Mb_{g,n}) \to H^{11}(\tilde{\partial \M}_{g,n})$ is injective.
\end{lem}
\begin{proof}
If $2g - 2 + n > 11$, this follows by combining Lemma \ref{injective} and Proposition \ref{compactsupport}.
If $2g - 2 + n \leq 11$ and $g \geq 2$, then \cite[Theorem 1.4]{CL-CKgP} shows $H^{11}(\Mb_{g,n}) = 0$.
\end{proof}

\begin{prop}\label{s12}
For any $g$, there is an injection \[H^{11}(\Mb_{g,n})\hookrightarrow \bigoplus H^{11}(\Mb_{1,11}).\]
\end{prop}
\begin{proof}
The result is known for $g \leq 1$, so we may assume $g \geq 2$.  By Lemma \ref{0to1}, we have an injection
\[
H^{11}(\Mb_{g,n})\hookrightarrow H^{11}(\tilde{\partial \M}_{g,n}).
\]
Each component of $\tilde{\partial \M}_{g,n}$ is a quotient by a finite group of
$\Mb_{g-1,n+2}$ or $\Mb_{g_1,n_1+1}\times\Mb_{g_2,n_2+1}$, where $k_1+k_2=n$, $2g_i-2+n_i>0$. 
Because $H^k(\Mb_{g',n'})=0$ for all $(g',n')$ and $k=1,3,5,7,9$ \cite{ArbarelloCornalba,BergstromFaberPayne}, the K\"unneth formula shows that
\begin{equation}\label{second}
H^{11}(\Mb_{g_1,n_1+1}\times\Mb_{g_2,n_2+1})=H^{11}(\Mb_{g_1,n_1+1})\oplus H^{11}(\Mb_{g_2,n_2+1}).
\end{equation}
Note that either $g_1<g$ or $g_1=g$ and $n_1+1<n$, and analogously for $(g_2,n_2)$. Therefore, there is an injective morphism of Hodge structures from $H^{11}(\Mb_{g,n})$ into a direct sum of Hodge structures of the form $H^{11}(\Mb_{\gamma,\nu})$ where $\gamma<g$ or $\gamma=g$ and $\nu<n$. By double induction on $g$ and $n$, and using that $H^{11}(\Mb_{0,n}) = 0$ and $H^{11}(\Mb_{1,n}) = 0$ for $n \leq 10$, we conclude that there is an injective morphism of Hodge structures or $\ell$-adic Galois representations
\[
H^{11}(\Mb_{g,n})\hookrightarrow \bigoplus H^{11}(\Mb_{1,11}). \qedhere
\]
\end{proof}

\section{An induction via the weight spectral sequence}\label{improved}
In this section, we prove the vanishing of $H^{11}(\Mb_{g,n})$ for $g \geq 2$. 
We do so by identifying the injection $H^{11}(\Mb_{g,n}) \hookrightarrow H^{11}(\tilde{\partial \M}_{g,n})$ (Lemma \ref{0to1}) as the first map in a complex and showing 
that the next map in the complex is injective.
The complex we use is obtained from the $E_1$-page in the weight spectral sequence associated to the compactification $\Mb_{g,n}$ of $\M_{g,n}$ \cite{DeligneHodgeII}, via Poincar\'e duality.  It is also the weight 11 summand of the Feynman transform of the modular operad $H^*(\Mb_{g,n})$ \cite{GetzlerKapranov}, and therefore has a natural graph complex interpretation \cite[Section~2.3]{PayneWillwacher21}.

\subsection{The first two maps in the weight $k$ complex}

Let $\Gamma$ denote a stable $n$-marked graph of genus $g$; the underlying graph is connected, each vertex $v$ is labeled by an integer $g_v$, and the valence of the vertex $v$ is denoted $n_v$. The stability condition is that $2g_v+n_v-2 > 0$. Set $\Mb_{\Gamma}=\prod_v \Mb_{g_v,n_v}$. There is the natural gluing map
\[
\xi_{\Gamma}\colon\Mb_{\Gamma}\rightarrow \Mb_{g,n}.
\]
The normalization of $\xi_{\Gamma}(\Mb_{\Gamma})$ is isomorphic to $\Mb_{\Gamma}/\Aut(\Gamma)$.
Thus, the first map in the weight $k$ complex, the pullback of $H^k$ to the normalization of the boundary, can be rewritten as:
\begin{equation}\label{firstmapgraphs}
H^{k}(\Mb_{g,n})\xrightarrow{ \ \alpha \ } \bigoplus_{|E(\Gamma)|=1} H^{k}(\Mb_{\Gamma})^{\Aut \Gamma}.
\end{equation}

\noindent The target of the $j$th map in the weight $k$ complex is $\bigoplus_{|E(\Gamma)|=j} (H^{k}(\Mb_{\Gamma}) \otimes \det E(\Gamma))^{\Aut \Gamma}.$ Here, $\det E(\Gamma)$ denotes the determinant of the permutation representation of $\Aut(\Gamma)$ acting on the set of edges.  We will only need the first and second maps. 

Let us describe the second map. We define
\begin{equation}\label{beta}
\bigoplus_{|E(\Gamma) = 1|} H^{k}(\Mb_{\Gamma})^{\Aut(\Gamma)} \xrightarrow{ \ \beta \ } \bigoplus_{|E(\Gamma) = 2|} (H^{k}(\Mb_{\Gamma}) \otimes \det E(\Gamma))^{\Aut(\Gamma)} 
\end{equation}
as follows.
For each graph $\Gamma$ with two edges, choose an ordering of 
the edges $E(\Gamma)= \{e_1, e_2\}$ and say $e_i$ corresponds to the node obtained by gluing the marked points $p_i$ and $q_i$. Let $\phi_i(\Gamma)$ denote the graph with one edge obtained by contracting $e_i$. Gluing $p_i$ and $q_i$ induces a map 
$\xi_i\colon \Mb_{\Gamma} \to \Mb_{\phi_i(\Gamma)}$, which in turn gives a map $\xi_i^*\colon H^{k}(\Mb_{\phi_i(\Gamma)}) \to H^{k}(\Mb_{\Gamma})$. This induces 
\begin{equation} \label{a1} 
H^{11}(\Mb_{\phi_i(\Gamma)}) \xrightarrow{\xi_i^*}
H^{11}(\Mb_{\Gamma}).
\end{equation}
 If $E(\Gamma)$ is a non-trivial representation of $\Aut(\Gamma)$, then $\xi_2 = \xi_1 \circ \sigma$ where $\sigma\colon \Mb_{\Gamma} \to \Mb_{\Gamma}$ is the automorphism that simultaneously swaps $p_1$ with $p_2$ and $q_1$ with $q_2$ (corresponding to the automorphism of $\Gamma$ that swaps the two edges). In particular, it follows that the image of $\xi_1^* - \xi_2^*$ lies in the subspace $(H^{k}(\Mb_{\Gamma}) \otimes \det E(\Gamma))^{\Aut(\Gamma)} \subset H^{k}(\Mb_{\Gamma})$.
Then $\beta$ is defined by taking the sum over all $2$-edge graphs. Then $\beta \circ \alpha = 0$, because each component is the difference of the pullbacks under two copies of the same gluing maps.

\subsection{Proof of Theorem \ref{H11}, assuming it holds for $g = 2$}

We first give a short inductive proof of Theorem \ref{H11}, assuming that $H^{11}(\Mb_{2,n}) = 0$ for all $n$. Fix $k = 11$ and $g \geq 3$.  By Lemma \ref{0to1}, the map $\alpha$ in \eqref{firstmapgraphs} is injective. We claim that the map 
$\beta$ in \eqref{beta} is also injective. The theorem follows from this claim, since $\beta \circ \alpha = 0$.

Consider the domain of $\beta$, which is $\bigoplus_{|E(\Gamma) = 1|} H^{k}(\Mb_{\Gamma})^{\Aut(\Gamma)}$.  If $\Gamma$ has a single vertex with a loop edge, then $H^{11}(\Mb_\Gamma)^{\Aut(\Gamma)} = H^{11}(\Mb_{g-1,n+2})^{\ss_2}$, which vanishes by induction on $g$.

Suppose $\Gamma$ has two vertices joined by an edge. Then $\Mb_\Gamma = \Mb_{a, A \cup p} \times \Mb_{b, A^c \cup q}$. By the K\"unneth formula and the vanishing of lower degree odd cohomology, 
\begin{equation} \label{eq:Kuenneth}
H^{11}(\Mb_{a, A \cup p} \times \Mb_{b,A^c \cup q})=H^{11}(\Mb_{a, A \cup p})\oplus H^{11}(\Mb_{b,A^c \cup q}).
\end{equation}
Assume $a \geq b$. By induction on $g$ and $n$, $H^{11}(\Mb_\Gamma)$ vanishes unless $b = 1$ and $|A^c| \geq 10$, in which case it is $H^{11}(\Mb_{1,A^c \cup q})$.  Note that, in this case, $\Aut(\Gamma)$ is trivial.

Let $\Gamma'$ be the graph obtained by attaching a loop to the vertex of genus $g-1$ (and decreasing its genus accordingly), as shown in Figure~\ref{Gamma2p}. 
\begin{figure}[h!]
    \centering
    \includegraphics[width=4.0in]{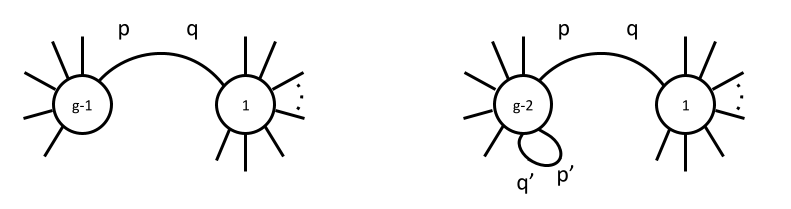}
    \caption{The graph $\Gamma$ on the left and $\Gamma'$ on the right. 
    }
    \label{Gamma2p}
\end{figure}As in \eqref{eq:Kuenneth}, we know that $H^{11}(\Mb_{\Gamma'})$ contains $H^{11}(\Mb_{1,A^c \cup q})$ as a summand. Then $\beta$ maps $H^{11}(\Mb_{\Gamma})$ injectively into this summand of $H^{11}(\Mb_{\Gamma'})$, and it follows that $\beta$ is injective, as claimed. \qed

\subsection{Proof of Theorem \ref{H11} for $g = 2$} \label{sec:g=2}

The proof that $H^{11}(\Mb_{g,n}) = 0$ for all $n$ follows a similar strategy to the inductive argument for $g\geq 3$. By Lemma \ref{0to1}, we know that $\alpha$ is injective, and we claim that $\beta$ is also injective.  The theorem follows from this claim.  However, the proof that $\beta$ is injective is more involved in this case.

We begin by describing some of the components of $\beta$ as concretely as possible.

\begin{example} Suppose $\Gamma$ is the leftmost graph in Figure \ref{graphs2}.
In this case, $\phi_1(\Gamma)$ and $\phi_2(\Gamma)$ are not isomorphic, so each component of $\beta|_{H^{11}(\Mb_\Gamma)}$ is one of the usual restriction maps, up to sign.
\end{example}

\begin{example}\label{example2} Suppose $\Gamma$ is the rightmost graph in Figure \ref{graphs2}. 
In this case, $\phi_1(\Gamma) \cong \phi_2(\Gamma)$. Let $\sigma\colon \Mb_{\Gamma} \to \Mb_{\Gamma}$ be the automorphism that simultaneously swaps $p_1$ with $p_2$ and $q_1$ with $q_2$ (corresponding to the automorphism of $\Gamma$ swapping the two edges). Then $\xi_2 = \xi_1 \circ \sigma$. By Proposition \ref{therep}, we know $\sigma^*\colon H^{11}(\Mb_{\Gamma}) \to H^{11}(\Mb_{\Gamma})$ is multiplication by $-1$. 
It follows that $\beta|_{H^{11}(\Mb_{\Gamma})}$ is again one of the usual restriction maps, up to rescaling. 
\end{example}

We will show that, for each $1$-edge graph $\Gamma$ there is a collection of $2$-edge graphs $\{\Gamma_i\}$ such that $\beta$ maps $H^{11}(\Mb_{\Gamma})^{\Aut(\Gamma)}$ injectively into $\bigoplus_i (H^{11}(\Mb_{\Gamma_i}) \otimes \det(E_{\Gamma_i}))^{\Aut(\Gamma_i)}$.  Moreover, we order the $1$-edge graphs in such a way that, at each step, none of the $2$-edge graphs $\{\Gamma_i\}$ admit edge contractions to any of the $1$-edge graphs that came earlier in the order.  In this way, we see that $\beta$ can be represented by a block upper diagonal matrix with injective blocks, and hence is injective, as required.

\medskip

For $n \leq 9$, the cohomology $H^*(\Mb_{2,n})$ is tautological \cite{CL-CKgP}, and hence $H^{11}(\Mb_{2,n}) = 0$, as required. The cases $n = 10$ and $11$ can be handled by an argument similar to that used for the cases $n \geq 12$, below, but the details are more involved.  Instead, we note that $\Mb_{2,n}$ is rational in these cases \cite{Casnati}, and hence, $H^{11,0}(\Mb_{2,n}) = 0$. By Proposition \ref{s12}, it follows that  $H^{11}(\Mb_{2,n}) = 0$. For the remainder of this subsection, we therefore assume $n \geq 12$. 

There are three types of $1$-edge graphs to consider: those with a vertex of genus $2$, those with a unique vertex of genus $1$ and a self-edge, and those with two vertices of genus $1$. 

\subsubsection{Graphs with a genus $2$ vertex}
Let $\Gamma$ be a graph with one edge and a genus $2$ vertex. Then $\Mb_{\Gamma}  = \Mb_{2,A \cup p} \times \Mb_{0, A^c \cup q}$. Because $|A \cup p| < n$, we see that $H^{11}(\Mb_{2,A \cup p}) = 0$ by induction. By the K\"unneth formula and the vanishing of lower degree odd cohomology, $H^{11}(\Mb_{\Gamma})=0$.

\subsubsection{Graphs with a single genus 1 vertex} \label{in2} 
Let $\Gamma$ be the graph with one vertex of genus $1$, $n$ legs, and one self-edge.
Consider $\Mb_{\Gamma} = \Mb_{1,n+2}$ labeled $\{1, \ldots, n, p_1, q_1\}$ where $p_1, q_1$ are glued together to get curves with dual graph $\Gamma$. 
As in Section \ref{ga}, given a subset $B \subset \{1, \ldots, n, p_1, q_1\}$ of cardinality $10$, we let $D_B = \Mb_{1,B \cup p_2} \times \Mb_{0, B^c \cup q_2}$. (This is nonempty because $n \geq 12$.)
There are four flavors of subsets $B$:
both of $p_1, q_1$ are in $B$; neither $p_1$ nor $q_1$ is in $B$; only $p_1$ is in $B$; or only $q_1$ is in $B$.
These correspond to $4$ types of codimension $1$ strata under the map $D_B \to \Mb_{1,n+2}$ that glues $p_2$ and $q_2$ (see Figure \ref{graphs1}).
\begin{figure}[h!]
    \centering
    \includegraphics[width=5.75in]{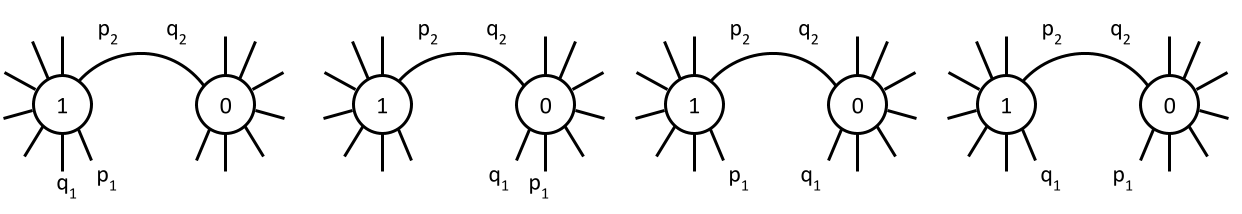}
    \caption{The four flavors of $D_B$.}
    \label{graphs1}
\end{figure}

For each subset $B$, let $\Gamma_B$ be the two-edge graph obtained by further gluing $p_1$ and $q_1$ (see Figure \ref{graphs2}). Note that $\Mb_{\Gamma_B} = D_B$. We want to show that
\[
H^{11}(\Mb_{\Gamma})^{\Aut(\Gamma)}\rightarrow \bigoplus_{|B|=10} H^{11}(\Mb_{\Gamma_B}\otimes \det E(\Gamma_B))^{\Aut(\Gamma_B)}
\]
is injective.

When we glue $p_1$ and $q_1$, the last two types of subsets $B$ give the same type of $2$-edge graph $\Gamma_B$: 
\begin{figure}[h!]
    \centering
    \includegraphics[width=5.75in]{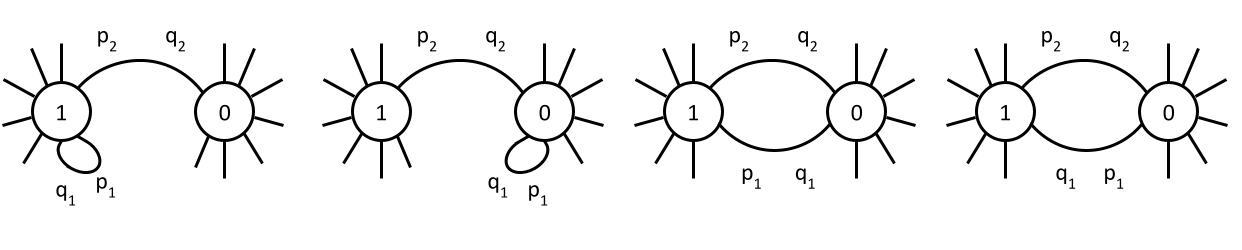}
    \caption{Gluing $p_1$ and $q_1$.}
    \label{graphs2}
\end{figure}

\noindent
Swapping $p_1$ and $q_1$ preserves the first two maps $D_B \to \Mb_{1,n+2}$ but exchanges the other two. Correspondingly, the edge representation is trivial in the first two cases and non-trivial in cases three and four.

By Corollary~\ref{cor:inj}, we have an injection
\[H^{11}(\Mb_{1,n+2}) \hookrightarrow \bigoplus_{\substack{|B| = 10}} H^{11}(D_B).\]
Now, $\ss_2$ acts on both sides by swapping $p_1$ and $q_1$. Taking $\ss_2$-invariants of both sides, we have an injection from $H^{11}(\Mb_{1,n+2})^{\ss_2}$ into 
\[ \bigoplus_{\substack{\{p_1, q_1\} \subset B \text{ or } \{p_1,q_1\} \subset B^c \\ |B| = 10} } H^{11}(D_B)^{\ss_2} \oplus \bigoplus_{\substack{p_1 \in B \text{ and } q_1 \notin B \\ |B| = 10}}[H^{11}(D_B) \oplus H^{11}(D_{(B\smallsetminus p_1) \cup q_1})]^{\ss_2}.\]
The first collection of terms is $(H^{11}(\Mb_{\Gamma_B}) \otimes \det E(\Gamma_B))^{\ss_2}$ for the graphs $\Gamma_B$ of the first and second flavor in Figure \ref{graphs2}, which have trivial edge representation.
For the terms in square brackets, we have $H^{11}(D_B) \cong H^{11}(D_{(B\smallsetminus p_1)\cup q_1})$ and the $\ss_2$-action switches the two factors. Hence, the space of $\ss_2$-invariants is $H^{11}(\Mb_{\Gamma_B})$. 
When $p_1 \in B$ and $q_1 \notin B$, note that $\Gamma_B$ is the type of graph considered in Example \ref{example2}. In particular $\Aut(\Gamma_B)$ acts by the sign representation, so
$(H^{11}(\Mb_{\Gamma_B}) \otimes \det E(\Gamma_B))^{\Aut(\Gamma_B)} = H^{11}(\Mb_{\Gamma_B})$.
In summary, we have given an injection
 \begin{equation} \label{summary} H^{11}(\Mb_{\Gamma})^{\Aut(\Gamma)} = H^{11}(\Mb_{1,n+2})^{\ss_2} \hookrightarrow \bigoplus_{|B| = 10} (H^{11}(\Mb_{\Gamma_B}) \otimes \det E(\Gamma_B))^{\Aut(\Gamma_B)}. 
 \end{equation}

\subsubsection{Graphs with two genus $1$ vertices}
Suppose $\Gamma$ has two genus $1$ vertices, so \[\Mb_{\Gamma} = \Mb_{1, A \cup p} \times \Mb_{1, A^c \cup q}.\] 
We may assume $|A^c| \leq |A|$. Figure \ref{graphs3} shows three types of graphs that will be used in the three cases of the argument. Note that none of them have a loop, so none admit edge contractions to the $1$-edge graphs of Section \ref{in2}.

\begin{figure}[h!]
    \centering
    \includegraphics[width=6.4in]{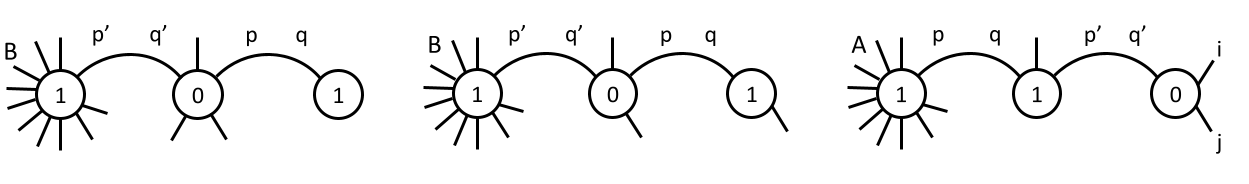}
    \caption{The graphs in Cases 1, 2, and 3.}
    \label{graphs3}
\end{figure}

\textit{Case 1: $|A^c| = 0$.} By Corollary~\ref{cor:inj}, there is an injection
\[H^{11}(\Mb_{\Gamma}) = H^{11}(\Mb_{1,A \cup p}) \hookrightarrow \bigoplus_{\substack{B \subset A \\ |B| = 10}} H^{11}(\Mb_{1, B \cup p'}) = H^{11}(\Mb_{\Psi_B}),\]
where $\Psi_B$ is the first graph in Figure \ref{graphs3}, which has
\[\Mb_{\Psi_B} = \Mb_{1, B \cup p'} \times \Mb_{0, (A \smallsetminus B) \cup \{q',p\}} \times \Mb_{1, \{q\}}.\]

\medskip
\textit{Case 2: $|A^c| = 1$.}
Again, by Corollary~\ref{cor:inj}, there is an injection
\[H^{11}(\Mb_{\Gamma}) = H^{11}(\Mb_{1,A \cup p}) \hookrightarrow \bigoplus_{\substack{B \subset A \\ |B| = 10}} H^{11}(\Mb_{1, B \cup p'}) = H^{11}(\Mb_{\Psi_B'}),\]
where this time $\Psi_B'$ is the second graph in Figure \ref{graphs3}, which has
\[\Mb_{\Psi_B'} = \Mb_{1, B \cup p'} \times \Mb_{0, (A \smallsetminus B) \cup \{q',p\}} \times \Mb_{1, A^c \cup q}.\]
Note that our assumptions $n \geq 12$ and $|A^c| = 1$ ensure $|A \smallsetminus B| \geq 1$, so the middle genus $0$ vertex has at least $3$ markings. Note that no edge contraction of $\Psi_B'$ gives a $1$-edge graph $\Gamma$ of the type in Case 1 or a different $\Gamma$ of the type in Case 2.

\medskip

\textit{Case 3: $|A^c| \geq 2$.}
Choose any $i, j \in A^c$ and $k, \ell \in A$  and define 
\[X_{A^c \cup q} = \Mb_{1,(A^c \smallsetminus \{i,j\}) \cup \{q, p'\}} \times \Mb_{0, \{q',i,j\}} \qquad \text{and} \qquad X_{A \cup p} = \Mb_{1,(A \smallsetminus \{k,\ell\}) \cup \{p, p'\}} \times \Mb_{0, \{q',k,\ell\}}.\]
Associated to these are $2$-edge graphs $\Phi_A$ and $\Phi_{A^c}$ so that
\[\Mb_{\Phi_A} = \Mb_{1, A \cup p} \times X_{A^c \cup q} \qquad \text{and} \qquad \Mb_{\Phi_{A^c}} = X_{A \cup p} \times \Mb_{1, A^c \cup q}\]
and
contracting the edge connecting $p'$ and $q'$ gives $\Gamma$. See the last graph in Figure \ref{graphs3} for a picture of $\Phi_A$. Notice that no edge contraction of $\Phi_A$ or $\Phi_{A^c}$ gives a $1$-edge graph appearing in Cases 1 or 2 above or a different $1$-edge graph of the type in Case 3.

By the K\"unneth formula, 
\begin{align*}
H^{11}(\Mb_{\Gamma}) &=H^{11}(\Mb_{1,A\cup p})\oplus H^{11}(\Mb_{1,A^c\cup q}),
\\
H^{11}(\Mb_{\Phi_A})&=H^{11}(\Mb_{1,A\cup p})\oplus H^{11}(\Mb_{1,(A^c\smallsetminus \{i,j\})\cup \{q,p'\}}),
\intertext{
and 
}
H^{11}(\Mb_{\Phi_{A^c}})&=H^{11}(\Mb_{1,A^c\cup q})\oplus H^{11}(\Mb_{1,(A\smallsetminus \{k,\ell\})\cup \{p,p'\}}).
\end{align*}
It follows that $H^{11}(\Mb_{\Gamma})$ injects into the sum
\begin{equation*}
H^{11}(\Mb_{\Phi_A})\oplus H^{11}(\Mb_{\Phi_{A^c}}).
\end{equation*}

\vspace{-.25in}
\qed
\medskip

\begin{rem}\label{alternate2}
It should also be possible to deduce the vanishing of $H^{11}(\Mb_{2,n})$ from the work of Dan Petersen as follows. There are exact sequences
\[
H^{k-2}(\Mb_{1,n+2})(-1)\rightarrow H^k(\Mb_{2,n})\rightarrow W_k H^k(\M^{\mathrm{ct}}_{2,n})\rightarrow 0.
\]
Specializing to $k=11$, we see that 
\[
H^{11}(\Mb_{2,n})\cong W_{11} H^{11}(\M^{\mathrm{ct}}_{2,n}).
\]
By \cite[Theorem 2.1 and Remark 2.2]{Petersen}, there is an isomorphism \[
H^{11}(\M^{\mathrm{ct}}_{2,n})\cong \bigoplus_{p+q=11} H^{p}(\M_2^{\mathrm{ct}},A^q)\oplus H^{p}(\Sym ^2 \M_{1,1}, B^q)
\]
where $A^q$ and $B^q$ are certain direct sums of Tate twists of symplectic local systems. The cohomology of these local systems has been determined by Petersen in \cite{Petersenlocalsystems}. Petersen's work makes significant use of high-powered machinery, including mixed Hodge modules, perverse sheaves, the decomposition theorem for the map $\M_{2,n}^{\mathrm{ct}}\rightarrow \M_{2}^{\mathrm{ct}}$, and the Eichler--Shimura isomorphism concerning modular forms. The proof we present here is relatively elementary and highlights a  combinatorial perspective on the vanishing of $H^{11}(\Mb_{2,n})$.
\end{rem}

\section{Application to point counting}\label{pointcounts}
The weighted count of $\mathbb{F}_q$ points on a Deligne--Mumford stack $X$ is 
\[
\#X(\mathbb{F}_q):=\sum_{x\in |X(\mathbb{F}_q)|} \frac{1}{\# \Aut (x)},
\]
where $|X(\mathbb{F}_q)|$ denotes the set of isomorphism classes of the groupoid $X(\mathbb{F}_q)$. This point count is related to the trace of the Frobenius map $\Phi_q$ on cohomology by Behrend's Grothendieck--Lefschetz trace formula \cite[Theorem 3.1.2]{Behrend}:
\[
\#X(\mathbb{F}_q)=q^{\dim X} \sum_{k\geq 0} (-1)^k \tr \Phi_q | H^k_{\text{\'et}}(X_{\overline{\mathbb{F}}_q},\qq_{\ell}).
\]
Using this formula in the case $X=\Mb_{g,n}$ leads to the proof of Corollary \ref{almost}.
\begin{proof}[Proof of Corollary \ref{almost}]
Let $d = \dim \Mb_{g,n} = 3g - 3 + n$.
The eigenvalues of Frobenius acting on $H^k_{\text{\'et}}((\Mb_{g,n})_{\overline{\mathbb{F}}_q}, \qq_{\ell})$ are Weil numbers of weight $k$, meaning that under any embedding $\overline{\mathbb{Q}}_\ell \hookrightarrow \cc$ they have absolute value $q^{k/2}$ \cite{DeligneWeil}. Thus, the point count $\#\Mb_{g,n}(\mathbb{F}_q)$ is determined up to $O(q^{d - 13/2})$ by the eigenvalues of Frobenius on $H^{2d-k}(\Mb_{g,n})$ for $k \leq 12$. Theorem \ref{even} and Poincar\'e duality tells us that Frobenius acts on the even cohomology $H^{2d - k}(\Mb_{g,n})$ for $k \leq 12$ by $q^{d - k/2}$.
Moreover, the groups $H^{2d-k}(\Mb_{g,n})$ vanish for odd $k \leq 9$ by \cite{ArbarelloCornalba, BergstromFaberPayne}.
Thus, up to $O(q^{d-13/2})$, the only other contribution to $\#\Mb_{g,n}(\mathbb{F}_{q})$ is from $H^{2d-11}(\Mb_{g,n})$, which vanishes when $g \geq 2$ by Theorem \ref{H11} and Poincar\'e duality.
\end{proof}

\section{Higher odd cohomology groups}\label{higherodd}
In this section, we prove Theorem \ref{higher}. The main tool is the push-pull formula. This formula is proven for manifolds in \cite[Corollary 2.2]{Yamaguchi}. The proof for orbifolds or smooth Deligne--Mumford stacks goes through analogously.
\begin{lem}\label{pushpull}
Suppose $i\colon X\hookrightarrow Y$ is a closed embedding of codimension $d$ between smooth Deligne--Mumford stacks. Let $N_{X/Y}$ denote the normal bundle. Then for any cohomology class $\alpha\in H^*(X)$
\[
i^*i_*(\alpha)=c_d(N_{X/Y})\cup \alpha.
\]
\end{lem}

\begin{proof}[Proof of Theorem \ref{higher}]
Set $g=1 + \sum_{i=1}^{k} g_i$. Assume $\{g_i\}$ is ordered so that $g_1<\dots<g_k$.
It suffices to prove the result for $\Mb_{g,11-k}$, as the pull back maps 
\[
H^{i}(\Mb_{g,n})\rightarrow H^{i}(\Mb_{g,n+1})
\]
are injective for all $n$.
Consider the gluing morphism
\[
\xi\colon \Mb_{1,11}\times \Mb_{g_1,1}\times \cdots \times \Mb_{g_k,1} \rightarrow \Mb_{g,11-k}
\]
attaching the marked point on the $\Mb_{g_i,1}$ component to the $i$th marked point on $\Mb_{1,11}$. Let $\alpha\in H^{11,0}(\Mb_{1,11})$ be a nonzero holomorphic $11$-form. We will show that 
\[
\xi_*(\alpha\otimes 1\otimes \dots\otimes 1)\neq 0 \in H^{11+2k}(\Mb_{g,11-k}).
\]
It suffices to show that $\xi^*\xi_*(\alpha\otimes 1\otimes \cdots\otimes 1)$ is nonzero. If $g_1\neq 1$, set
\[
U:=\M_{1,11}\times \M_{g_1,1}\times \cdots \times \M_{g_k,1}.
\] 
If $g_1=1$, set 
\[
U:=\M_{1,11}\times \Mb_{1,1}\times \M_{g_2,1}\times\cdots \M_{g_k,1}.
\]
We will show that
\[
\xi^*\xi_*(\alpha\otimes 1\otimes \cdots\otimes 1)|_{U}\neq 0.
\]
Let $\Gamma_{g,11-k}$ denote the stable graph with $k+1$ vertices as follows. One vertex is of genus $1$ with $11-k$ half-edges, and the rest of the vertices are of genus $g_i$. There is exactly one edge between the first vertex and each of the latter vertices, and no other edges, unless $g_1=1$ in which case we allow ourselves to replace the $g_1$ vertex with a genus $0$ vertex with a self-edge (see Figure \ref{possibilities}).
\begin{figure}[h!]
    \centering
    \includegraphics[width=4.5in]{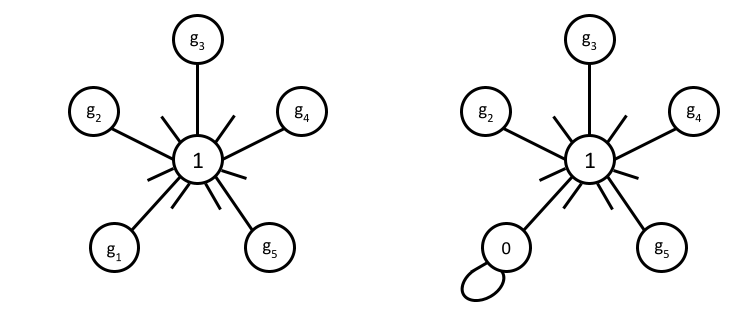}
    \caption{The graphs $\Gamma_{g,11-k}$.}
    \label{possibilities}
\end{figure}
Let $V_{g,11-k}$ denote the open substack of $\Mb_{g,11-k}$ parametrizing curves whose dual graphs are obtained from $\Gamma_{g,11-k}$ by edge contraction.
Then \[\xi|_{U}\colon U\rightarrow \Mb_{g,11-k}\]
factors through the open substack $V_{g,11-k}$. The induced map
\[
i\colon U\hookrightarrow V_{g,11-k}
\]
is a closed embedding between smooth stacks of codimension $k$. Here, we are using that the genera $g_1,\dots,g_k$ are distinct. Let $\beta$ denote the restriction of $\alpha$ to $\M_{1,11}$, which is nonzero (for example, by \cite{Getzler3}). By Lemma \ref{pushpull} 
\[
i^*i_*(\beta\otimes 1\otimes \cdots\otimes 1)=c_k(N_i)\cup (\beta\otimes 1\otimes \cdots \otimes 1).
\]
Let $\psi_{g_j}$ denote the (unique) $\psi$ class on the $(j+1)$st component of $U$. Let $\psi_j$ denote the $j$th $\psi$ class on $\M_{1,11}$. 
It is well-known that 
\[
c_k(N_i)=\prod_{j=1}^k(-\psi_j\otimes 1\otimes \cdots\otimes 1-1\otimes \cdots\otimes \psi_{g_j}\otimes\cdots\otimes 1).
\]
See, for example, the discussion in \cite[Section A.4]{GraberPandharipande}.
Because $H^{13}(\M_{1,11})=0$ \cite{Getzler3}, all products of the form $\beta\psi_{j}$ vanish. Therefore,
\[
c_k(N_i)\cup (\beta\otimes 1\otimes \cdots\otimes 1)=(-1)^{k}(\beta \otimes \psi_{g_1}\otimes \cdots\otimes \psi_{g_k})\neq 0. \qedhere
\]

\end{proof}
\section{Application to Chow rings}\label{Chow}
We denote by $A^*(X)$ the Chow ring of $X$ with $\qq$-coefficients.
\begin{thm}[Kimura \cite{Kimura}, Totaro \cite{TotaroCKgP}]\label{cycleclass}
Suppose that $X$ is a smooth, proper Deligne--Mumford stack over $\cc$. If $A^*(X)$ is a countable $\qq$-vector space, then the cycle class map
\[
A^*(X)\rightarrow H^{*}(X)
\]
is an isomorphism.
\end{thm}
\begin{proof}

First, note that $X$ is defined over a subfield $E \subset \cc$ that is finitely generated over $\qq$ and hence countable. Suppose $A^*(X)$ is countable. Then there is a countable extension $F$ of $E$ such that $A^*(X_{F})\rightarrow A^*(X)$ is surjective. Let $G$ be a finitely generated extension of $F$. Then $G$ can be embedded in $\cc$, so we have a morphism
\[
A^*(X_F)\rightarrow A^*(X_G)\rightarrow A^*(X).
\]
Each of the maps above is injective by
\cite[Proposition 3.2]{Kimura}.
Because the composite is surjective, it follows that the first morphism is also surjective. By \cite[Theorem 4.1]{TotaroCKgP}, it follows that the motive of $X_F$ is pure Hodge--Tate, and thus the motive of $X$ is as well. Note that \cite[Theorem 4.1]{TotaroCKgP} is stated for schemes, but the same proof goes through for Deligne--Mumford stacks. In particular, the cycle class map is an isomorphism.
\end{proof}
\begin{proof}[Proof of Corollary \ref{nontautological}]
The tautological ring $R^*(\Mb_{g,n})$ is a finite dimensional $\qq$-vector space by \cite[Corollary 1]{GraberPandharipande}. By Theorem \ref{cycleclass}, we know that $A^*(\Mb_{g,n})$ is uncountable whenever there is odd cohomology. Therefore, the quotient $A^*(\Mb_{g,n})/R^*(\Mb_{g,n})$ is uncountable for the values of $(g,n)$ in Theorem \ref{higher}. It follows that the subgroup of $A^*(\Mb_{g,n})$ generated by cycles algebraically equivalent to zero is also uncountable, since the Hilbert scheme has only countably many connected components.
\end{proof}

\bibliographystyle{amsplain}
\bibliography{refs}
\end{document}